\documentclass{amsart}

\usepackage{amsmath}
\usepackage{amssymb}
\usepackage{amscd}

\title[Lower bounds for cohomology growth]
 {On lower bounds for cohomology growth in $p$-adic analytic towers} 

\author[S. Kionke]{Steffen Kionke}

\address{Max Planck Institute for Mathematics \\ Vivatsgasse 7 \\ 53111 Bonn \\ Germany.}
\email{skionke@mpim-bonn.mpg.de}
\thanks{The author would like to thank the Max Planck Institute for Mathematics in Bonn for their hospitality and support.}
\date{\today}

\theoremstyle{plain}
\newtheorem{lemma}{Lemma}
\newtheorem{theorem}{Theorem}

\newtheorem{proposition}{Proposition}

\theoremstyle{definition}
\newtheorem{definition}{Definition}
\newtheorem{remark}{Remark}
\newtheorem{example}{Example}

\DeclareMathOperator{\Tr}{Tr}

\DeclareMathOperator{\Ad}{Ad}

\DeclareMathOperator{\inn}{int}
\DeclareMathOperator{\Lie}{Lie}
\DeclareMathOperator{\GL}{GL}
\DeclareMathOperator{\diag}{diag}
\DeclareMathOperator{\vcd}{vcd}
\DeclareMathOperator{\SL}{SL}

\DeclareMathOperator{\Res}{Res}
\DeclareMathOperator{\rank}{rank}
\DeclareMathOperator{\sign}{sign}

\providecommand{\typ}[1]{({#1})}

\providecommand{\calL}{\mathcal{L}}
\providecommand{\cO}{\mathcal{O}}

\providecommand{\up}[1]{\,^{#1}}
\providecommand{\isomorph}{\stackrel{\simeq}{\longrightarrow}}
\providecommand{\semidirect}{\rtimes}
\providecommand{\Gm}{\mathbb{G}_m}

\providecommand{\Ip}{\mathfrak{p}}
\providecommand{\Ia}{\mathfrak{a}}

\providecommand{\LG}{\mathfrak{g}}
\providecommand{\LZ}{\mathfrak{z}}
\providecommand{\LK}{\mathfrak{k}}
\providecommand{\LP}{\mathfrak{p}}

\providecommand{\bbR}{\mathbb{R}}
\providecommand{\bbQ}{\mathbb{Q}}
\providecommand{\bbZ}{\mathbb{Z}}
\providecommand{\bbF}{\mathbb{F}}
\providecommand{\bbC}{\mathbb{C}}
\providecommand{\normal}{\trianglelefteq}

\begin{document}
\begin{abstract}
 Let $p$ and $\ell$ be two distinct prime numbers and let $\Gamma$ be a group.
 We study the asymptotic behaviour of the mod-$\ell$ Betti numbers in $p$-adic analytic towers of
 finite index subgroups.
 If $\Theta$ is a finite $\ell$-group of automorphisms of $\Gamma$, our main theorem allows to
 lift lower bounds for the mod-$\ell$ cohomology growth in the fixed point group $\Gamma^\Theta$ to lower bounds
 for the growth in $\Gamma$.
 We give applications to $S$-arithmetic groups and we also
 obtain a similar result for cohomology with rational coefficients.
\end{abstract}

\maketitle

\section{Introduction}

The study of (co)homology growth of a group $\Gamma$ is concerned with the question how the 
(co)homology and associated invariants change in towers of finite index (normal) subgroups.
 In recent years (co)homology growth has become a rather active topic with contributions by
L\"uck \cite{Lueck1994, Lueck2013}, Lackenby \cite{Lackenby2009a, Lackenby2009b} and
Calegari-Emerton \cite{CalegariEmerton2009, CalegariEmerton2011}.
In particular, Calegari and Emerton proved upper bound results for mod-$p$ cohomology growth in $p$-adic analytic towers 
(see also Bergeron, Linnell, L\"uck and Sauer \cite{BLLS2012} for a different perspective).

Throughout the article $p$ and $\ell$ denote two distinct prime numbers.
The main result in this article provides a simple way to prove asymptotic lower bounds for 
the mod-$\ell$ Betti number growth in $p$-adic analytic towers of subgroups.

\subsection{The main result}
Let $\Gamma$ be a group of type \typ{VFP}, by definition this means that
every torsion-free finite index subgroup is of type \typ{FP}.
Recall that a group $\Gamma'$ is of type \typ{FP} if the trivial group module $\bbZ$ admits a finite length
resolution by finitely generated projective $\bbZ[\Gamma']$-modules (see \cite[p.199]{BrownBook1982}).
In particular, every torsion-free finite index subgroup of $\Gamma$ has finite mod-$\ell$ cohomology.
Assume there is an injective group homomorphism $\iota: \Gamma \to G$, where $G$ is a compact $p$-adic analytic group.
We will assume further that the image of $\iota$ is dense in $G$.
We fix a finite $\ell$-group $\Theta$ of continuous automorphisms of $G$ and we assume that the image $\iota(\Gamma)$
is stable under $\Theta$. Thus $\Theta$ also acts by automorphisms on the group~$\Gamma$.
From now on we identify $\Gamma$ with its image under $\iota$.
We fix an open, normal, $\Theta$-stable, uniform pro-$p$ subgroup $U \normal G$ (see Lemma \ref{lem:exU} below)
 and we consider the associated lower $p$-series of $U$ defined by
\begin{equation*}
  P_1(U) := U \quad\text{ and }\quad P_{i+1}(U) := \Phi(P_i(U)).
\end{equation*}
Here $\Phi$ to denotes the Frattini subgroup. For example, if $G = \GL_n(\bbZ_p)$ (with $p$ odd), then the principal congruence subgroup $U$
of level one is uniform and $P_i(U)$ is the principal congruence subgroup of level $i$.
For every $i\geq 1$ we define 
    \begin{equation*}
       \Gamma_i := \Gamma \cap P_i(U).
    \end{equation*}
We study the mod-$\ell$ cohomology growth for the tower of torsion-free, normal, finite index subgroups $(\Gamma_i)_i$ of $\Gamma$.
Inspired by \cite{BLLS2012} we shall call this a $p$-adic analytic tower. 
Note that for every $i$ the group $\Gamma_i$ is $\Theta$-stable since the subgroups $P_i(U)$ are characteristic in $U$.
\begin{theorem}\label{thm:modGrowth}
   Let $r \geq 0$ be an integer.
   Suppose there are positive real numbers $\alpha \in [0,1]$ and $\lambda \in [0,\dim(G^\Theta)]$ such that 
   \begin{equation*}
       \sum_{i=r}^\infty \dim_{\bbF_\ell} H^i(\Gamma^\Theta_n, \bbF_\ell) \gg \Bigl(\frac{[\Gamma^\Theta : \Gamma^\Theta_n]}{p^{(n-1)\lambda}}\Bigr)^\alpha,
   \end{equation*}
   as $n$ tends to infinity, then 
    \begin{equation*}
       \sum_{i=r}^\infty \dim_{\bbF_\ell} H^i(\Gamma_n, \bbF_\ell) \gg [\Gamma : \Gamma_n]^{d\alpha}
   \end{equation*}
   where $d = \frac{\dim(G^\Theta)-\lambda}{\dim(G)}$ as $n$ tends to infinity\footnote{Here ``$\gg$`` means the inequality ``$\geq$`` holds for large $n$ up to a positive constant.}.
   In particular, the conclusion applies with $r=0$, $\alpha=1$ and $\lambda=0$ whenever $\Gamma^\Theta$ is a finite group.
\end{theorem}
To obtain strong bounds out of this result we would like to choose $\alpha$ large and $\lambda$ small. Afterwards
we always try to choose $r$ as large as possible to obtain more information on the degrees in which the Betti numbers grow.

The proof is based on Smith theory and its group theoretic interpretation by Adem. 
We sketch the argument -- a full proof is given in Section \ref{sec:proof}.
Choose a suitable $K(\Gamma_n,1)$ space $X_n$ which is equipped with an action of the semidirect product $\Gamma_n\rtimes\Theta$. 
The basic inequality of Smith-Floyd \cite{Floyd1952} implies
\begin{equation*}
   \sum_{i=r}^\infty \dim H^i(X_n,\bbF_\ell) \geq \sum_{i=r}^\infty \dim H^i(X_n^\Theta,\bbF_\ell).
\end{equation*}
Adem gave a group theoretic interpretation of the space of fixed points (cf.~\cite{Adem1996}), which yields
\begin{equation*}
   \dim H^i(X_n^\Theta,\bbF_\ell) = \sum_{[c] \in H^1(\Theta,\Gamma_n)} \dim H^i(\Gamma_n^{\Theta|c},\bbF_\ell),
\end{equation*}
where the sum runs over the first non-abelian Galois cohomology set of $\Theta$ with values in $\Gamma_n$.
Here $\Gamma_n^{\Theta|c}$ denotes the subgroup fixed by the $c$-twisted action of $\Theta$ (see Section~\ref{sec:proof}).
We do not assume that the index of the fixed point groups actually grows as $n$ tends to infinity. Instead, 
we show that the cardinality of
the cohomology set $H^1(\Theta,\Gamma_n)$ grows sufficiently fast.
The underlying results are obtained in Section~\ref{sec:prop} using pro-$p$ methods.
Here we use the assumption that the tower is $p$-adic analytic.

In Section \ref{sec:applications} we give two applications of Theorem~\ref{thm:modGrowth} concerning the cohomology of $S$-arithmetic groups.
In the first application (Theorem~\ref{thm:basechange}) we lift cohomology growth of principal congruence subgroups
 over base change of Galois extensions of order $2^m$.
In the second application (Theorem~\ref{thm:splitGroups}) we give a uniform lower bound for the mod-$\ell$
 cohomology growth of $S$-arithmetic subgroups of split semi-simple groups.
In Section \ref{sec:rational} we investigate the question whether the cohomology growth obtained via Theorem~\ref{thm:modGrowth}
is already attained for the growth of cohomology with rational coefficients.
In fact, under stronger hypotheses we obtain a similar result (Theorem~\ref{thm:Qgrowth}) concerning the rational cohomology
 growth of arithmetic groups. 

\subsection{Smith theory and non-constant coefficients}
Since there is a Smith theory for sheaf cohomology (cf.~Thm.~19.7 in \cite{Bredon1997}), it
is tempting to ask for a generalization of Smith theory (and Theorem \ref{thm:modGrowth}) for group cohomology with non-constant coefficients.
Clearly, given a finite dimensional $\bbF_\ell$ vector space $M$ with an action of $\Gamma$, there is a finite index normal subgroup
of~$\Gamma$ which acts trivially. Passing to this subgroup (and possibly shrinking to make it $\Theta$-stable) we can simply apply 
Theorem \ref{thm:modGrowth}.
However, if we do not allow passing to a finite index subgroup first, then we cannot
apply Smith theory, since necessary conditions are in general not fulfilled 
by sheaves induced by a non-constant group module on the Eilenberg-MacLane space.

Consider the following counter-example.
Let $\Gamma = \bbZ$ and let $\Theta = \bbZ/2\bbZ$ be acting on $\Gamma$ by inversion, in particular
$\Gamma^\Theta = \{0\}$.
The group $\Gamma$ acts by cyclic permutations on the $\bbF_2$ vector space
\begin{equation*}
   M = \{\:(x,y,z) \in \bbF_2 \:|\: x + y + z = 0\:\}.
\end{equation*}
It is easy to check that $H^*(\bbZ, M) = 0$. However, $H^0(\Gamma^\Theta, M) = M$  and so the basic inequality of Smith theory
 $\dim H^*(\bbZ, M) \geq \dim H^*(\bbZ^\Theta, M)$ fails.
This gives a counter-example to the extremal case $d=0$ of Theorem \ref{thm:modGrowth} for non-constant coefficients.

\subsection{Further comments} 
Whereas strong upper bound results have been obtained recently (cf.~\cite{CalegariEmerton2009}, \cite{BLLS2012}),
lower bounds for cohomology growth seem to be a more complicated issue.
The best results are known for the homology of low dimensional manifolds.
Lackenby \cite{Lackenby2009b} studied the growth of the first mod-$p$ Betti number
and obtained strong asymptotic lower bounds under assumptions on the group $\Gamma$ which are inspired by $3$- and $4$-manifold topology.
Calegari and Emerton obtained precise growth results for the first mod-$p$ homology of compact $3$-manifolds (cf. \cite{CalegariEmerton2011}).
For higher dimensional cases it seems that there are no general results
(F.~Calegari pointed out that their basic technique should yield results for cocompact arithmetic lattices in general). There are some 
cases for arithmetic groups where the computation of Lefschetz numbers
 implies certain lower bounds (eg.~\cite{KSchwermer2012}).
Such applications can be obtained directly from our Theorem~\ref{thm:Qgrowth} without computation of Lefschetz numbers.
Moreover, for arithmetic groups there is also a fruitful connection to the theory of limit multiplicities and automorphic representations
(see Sarnak-Xue \cite{SarnakXue1991} for a conjecture on the asymptotic behaviour of multiplicities).
In this framework there are some lower bound results due to Rajan and Venkataramana \cite[Thm.6]{RajanVenky2001}.

\subsection{Notation and Conventions}\label{def:cofinaltower}
Throughout $p$ and $\ell$ denote two distinct prime numbers.
For convenience we make the following definitions.
\begin{enumerate}
 \item  A \emph{cofinal $p$-tower} in a group $\Gamma$ is a decreasing sequence of finite index subgroups 
$\Gamma_i \subseteq \Gamma$ such that the following two conditions hold:
\begin{enumerate}
 \item $\bigcap_i \Gamma_i = \{1\}$ and
 \item for all $n$ the group $\Gamma_n$ is normal in $\Gamma_1$ and $\Gamma_{n}/\Gamma_{n+1}$ is an elementary abelian $p$-group.
\end{enumerate}
 \item Let $G$ be a real semi-simple Lie group with a finite number of connected components and let $K$ be a maximal compact subgroup.
       We say that $G$ \emph{has discrete series} -- for simplicity  just \emph{DS} --
      if the ranks of the complexified Lie algebras of $G$ and $K$ agree.
\end{enumerate}

\section{Applications to arithmetic groups}\label{sec:applications}
The groups of type \typ{VFP} of particular interest in this article are $S$-arithmetic subgroups of semi-simple algebraic groups.
We demonstrate how Theorem~\ref{thm:modGrowth} can be used in this situation. In the first application we stick to arithmetic groups,
in the second application (Thm.~\ref{thm:splitGroups}) we treat $S$-arithmetic groups for sufficiently small sets $S$.

Let $F$ be an algebraic number field with ring of integers $\cO_F$ and let $G$ be a connected semi-simple linear algebraic $F$-group.
We fix an $F$-rational embedding $\phi: G \to \GL_N$.
We define the real Lie group $G_\infty = G(F \otimes_\bbQ \bbR)$ and we fix a maximal compact subgroup $K_\infty \subset G_\infty$.
Recall that we say $G_\infty$ has discrete series (short ``DS'') if the complex ranks $\rank_\bbC(K_\infty)$ and  $\rank_\bbC(G_\infty)$ agree.
As usual we put $q = 1/2 \cdot \dim(G_\infty/K_\infty)$.

For any proper ideal $\Ia \subset \cO_F$ we define the associated principal congruence subgroups
$\Gamma(\Ia) \subset G(F)$ by 
\begin{equation*}
     \Gamma(\Ia) = \phi^{-1}\bigl(\ker(\GL_N(\cO_F) \to \GL_N(\cO_F/\Ia))\bigr).
\end{equation*}
Recall the following important result:
For every prime ideal $\Ip \subset \cO_F$ one has
\begin{equation*}
  \lim_{k\to\infty} \frac{H^i(\Gamma(\Ip^k), \bbC)}{[\Gamma(\Ip):\Gamma(\Ip^k)]} = \begin{cases}
                                                                                      (-1)^q\chi(\Gamma(\Ip)) \quad &\text{ if $G_\infty$ has DS and $i = q$}\\
											      0 \quad &\text{ otherwise }.
                                                                                   \end{cases}
\end{equation*}
Note further that in this case the Euler characteristic $\chi(\Gamma(\Ip))$ is non-zero.
The result follows from the theory of limit multiplicites of de George-Wallach, Rohlfs-Speh and Savin. 
More precisely, one might use Thm.~3.4 in \cite{CalegariEmerton2009} to obtain the case where the limit is zero. Finally, one only needs
to use the fact that the Euler characteristic grows linearly with the index (cf.~\cite{Serre1971}).
In fact, by L\"uck's approximation theorem~\cite{Lueck1994} this result describes the $L^2$ Betti numbers of $\Gamma(\Ip)$.

Hence, if $G_\infty$ has DS, we have asymptotic lower bounds
\begin{equation}\label{eq:lowerbound}
    H^q(\Gamma(\Ip^k),\bbF_\ell) \gg [\Gamma(\Ip):\Gamma(\Ip^k)]
\end{equation}
for every prime number $\ell$. Theorem \ref{thm:modGrowth} makes it possible to lift this cohomology growth to other groups. 

In order to apply Theorem \ref{thm:modGrowth} we need to know whether the closure of an arithmetic group at some
finite place is actually an open compact subgroup.
For a prime ideal $\Ip \subset \cO_F$ we denote the completion of $F$ w.r.t.~$\Ip$ by $F_\Ip$.
\begin{lemma}\label{lem:openclosure}
  Let $G$ be a connected semi-simple linear algebraic group over $F$. 
  Let $\Gamma \subseteq G(F)$ be an arithmetic group.  
  If the simply connected covering $\widetilde{G}$ of $G$ has strong approximation,
  then the closure of $\Gamma$ in $G(F_\Ip)$ is an open compact subgroup.
\end{lemma}
\begin{proof}
 Let $\pi: \widetilde{G} \to G$ be the $F$-rational covering morphism.
 The derivative of $\pi_{F_\Ip}: \widetilde{G}(F_\Ip) \to G(F_\Ip)$ is the identity, therefore
 $\pi_{F_\Ip}$ is a submersion. In particular, it is an open map.
 Due to the commensurability of arithmetic groups it suffices to prove the claim for some arithmetic group $\Gamma_0$.
 Let $\Delta \subseteq \widetilde{G}(F)$ be an arithmetic group, by Cor.~7.13 in \cite{Borel1969} there
 is an arithmetic subgroup $\Gamma_0 \subset G(F)$ such that $\pi(\Delta) \subset \Gamma_0$.
 Since $\widetilde{G}$ has strong approximation, the closure $\overline{\Delta} \subseteq \widetilde{G}(F_\Ip)$ is open and compact.
 We obtain 
  \begin{equation*}
      \pi(\overline{\Delta}) = \overline{\pi(\Delta)} \subseteq \overline{\Gamma_0} \subset G(F_\Ip),
  \end{equation*}
   and the claim follows from the fact that $\pi_{F_\Ip}$ is open.
\end{proof}

\subsection{Application 1: Base change}
Let $E/F$ be a Galois extension of number fields with Galois group $\Theta$.
We consider the semi-simple algebraic $F$-group 
$H = \Res_{E/F}(G \times_F E)$ obtained by restriction of scalars
and we study congruence subgroups of $H$.
Note that $\Theta$ acts on $H$ by $F$-rational automorphisms.
Let $\Ip \subset \cO_F$ be a prime ideal,
the associated principal congruence subgroups 
\begin{equation*}
   \Gamma(\Ip, k) = \phi^{-1}\Bigl(\ker\bigl( \GL_N(\cO_E) \to \GL_N(\cO_E/\Ip^k\cO_E)\bigr)\Bigr) \subset H(F) = G(E)
\end{equation*}
are stable under $\Theta$.
Further, we define $\Gamma = H(\cO_F) = G(\cO_E)$ via the embedding $\phi$. 
\begin{theorem} \label{thm:basechange}
Assume that $[E:F] = 2^m$. 
  Let $\Ip \subset \cO_F$ be a prime ideal which does not divide $2$.
  If $G_\infty$ has DS and the simply connected covering of $H$ has strong approximation, then
   \begin{equation*}
       \sum_{i=q}^\infty \dim H^i(\Gamma(\Ip, k), \bbF_2) \gg [\Gamma: \Gamma(\Ip, k)]^{1/[E:F]}
   \end{equation*}
   as $k$ tends to infinity, where $q = \frac{1}{2}\dim(G_\infty/K_\infty)$.
\end{theorem}
\begin{proof}
   Let $p$ be the prime number such that $\Ip \cap \bbZ = p\bbZ$.
   We denote the completion of $F$ w.r.t.~the ideal $\Ip$ by $F_\Ip$, moreover we denote the valuation ring in $F_\Ip$ by~$\cO_\Ip$.
   The closure of $\Gamma(\Ip,k)$ in the $p$-adic analytic group $H(F_\Ip)$ is an open compact subgroup
   by Lemma \ref{lem:openclosure}.
   We define the open compact subgroup 
    \begin{equation*}
         K_\Ip(k) := \ker\bigl( H(\cO_\Ip)\to H(\cO_\Ip/\Ip^k\cO_\Ip)\bigr)
    \end{equation*}
   (always with respect to the embedding $\phi:G\to \GL_N$).
   These groups form a neighbourhood basis of open subgroups in $H(F_\Ip)$ and therefore 
   \begin{equation*}
        \overline{\Gamma(\Ip,k)} = K_\Ip(k)
   \end{equation*}
    for all sufficiently large $k$. Note that, if $H$ has strong approximation, then this holds for all $k$.
    Further, for all sufficiently large $k$ the p-adic analytic group $K_\Ip(k)$ is a uniform pro-$p$ group and we have the identity
    \begin{equation*}
        \Phi(K_\Ip(k)) = K_\Ip(k+e),
    \end{equation*}
     where $e$ is largest integer such that $\Ip^e$ divides the ideal $p\cO_F$. Note that $H^\Theta = G$.
     Henceforth we may apply Theorem~\ref{thm:modGrowth} to the groups $\Gamma(\Ip,k)$, at least for sufficiently large $k$.
     Here we choose $r = q$, $\alpha=1$ and $\lambda=0$, which is possible due to equation \eqref{eq:lowerbound} (using that $G_\infty$ has DS).
     Clearly, the fixed point group $H(F_\Ip)^\Theta$ is just $G(F_\Ip)$ and moreover  
     we observe 
     \begin{equation*}
         \dim(K_\Ip(k)) = \dim(H(F_\Ip)) = [E:F] \dim(G(F_\Ip)). \qedhere
     \end{equation*}
\end{proof}

The same proof works for mod-$\ell$ cohomology if $[E:F] = \ell^m$ for an odd prime~$\ell$. 
However, in this case $H$ has discrete series and the claim is clear.
For instance, this theorem can be applied to symplectic groups.
Of course it is possible to use a lower bound produced by Theorem~\ref{thm:basechange} 
and to lift it (via Thm.~\ref{thm:modGrowth}) yet to another group. 
We make this more precise with the following example.

\begin{example}\label{ex:SL}
Let $F$ be a totally real number field and let $E/F$ be a Galois extension of degree $[E:F] = 2^m$.
Consider the special linear group $\SL_{2n}/E$. Let $\Ip$ be a prime ideal of $\cO_F$ which does not divide $2$.
We study $\SL_{2n}(\cO_E)$ and the associated principal congruence subgroups
\begin{equation*}
    \Gamma(\Ip, k) := \ker\bigl(\SL_{2n}(\cO_E) \to \SL_{2n}(\cO_E/\Ip^k\cO_E)\bigr).
\end{equation*}
 In this situation one obtains
   \begin{equation*}
       \sum_{i = [F:\bbQ]\frac{n^2+n}{2}}^{\vcd(\SL_{2n}(\cO_E))} \dim_{\bbF_2} H^i(\Gamma(\Ip,k), \bbF_2)
     \gg [\SL_{2n}(\cO_E):\Gamma(\Ip,k)]^{\frac{n(2n+1)}{2^m(4n^2-1)}}
   \end{equation*}
   as $k$ tends to infinity.
   Here $\vcd(\SL_{2n}(\cO_E))$ denotes the virtual cohomological dimension of $\SL_{2n}(\cO_E)$.
   It follows from the work of Borel-Serre \cite[11.4.4]{BorelSerre1973} that the virtual cohomological dimension is given by  
   \begin{equation*}
       \vcd(\SL_{2n}(\cO_E)) = 2^{m+1}[F:\bbQ]n^2 + (s-2)n - (s+t-1).
   \end{equation*}
   Here $[E:\bbQ] = s + 2t$ where $s$ denotes the number of real places and $t$ the number of complex places of $E$.
   To obtain this result, one simply applies Theorem \ref{thm:modGrowth} for the finite group $\Theta$ generated by the standard symplectic involution and uses the lower bounds
    obtained from Theorem \ref{thm:basechange} for the symplectic group.   
\end{example}

\subsection{Application 2: Split groups}

Let $F$ be an algebraic number field. 
Let $s$ denote the number of real places of $F$ and $t$ the number of complex places of $F$, then $[F:\bbQ] = s + 2t$.
We fix a finite set of places $S$ of $F$ which contains the set of all archimedean places. 
Theorem \ref{thm:modGrowth} enables us to prove 
a general result on cohomology growth for $S$-arithmetic subgroups in split $F$-groups.
For certain groups (for instance if $S$ is the set of all archimedean places and $G_\infty$ has DS) this bound is too weak, however the bound is uniformly valid for all
split semi-simple groups.

\begin{theorem}\label{thm:splitGroups}
  Let $G$ be a semi-simple split $F$-group of rank $r$ and assume that the finite set of places $S$ satisfies $|S|\leq[F:\bbQ]$.
  Suppose $F$ contains all $\ell$-th roots of unity for some prime $\ell$ and choose any prime number $p \neq \ell$ which is prime to all finite places in $S$.
  Given any $S$-arithmetic subgroup $\Gamma \subset G(F)$ and a finite dimensional $\bbF_\ell[\Gamma]$-module $V$, 
there is a cofinal $p$-tower $(\Gamma_n)_{n=1}^\infty$ (cf.~\ref{def:cofinaltower}) 
  such that
   \begin{equation*}
       \sum_{i= r(|S|-1)}^\infty \dim_{\bbF_\ell} H^i(\Gamma_n, V) \gg [\Gamma:\Gamma_n]^{\alpha}
   \end{equation*}
   with $\alpha = \frac{r([F:\bbQ]-|S|+1)}{\dim(G)[F:\bbQ]}$ as $n$ tends to infinity.
\end{theorem}
\begin{proof}
   Passing to a finite index subgroup we can assume that $G$ is connected and that $\Gamma$ acts trivially on $V$.
   Let $\widetilde{G}$ denote the universal covering of $G$. Since $G$ is $F$-split, we see that $\widetilde{G}$ is split as well,
   hence it has strong approximation.
   Consider the associated adjoint group $\Ad(G) = G/C(G)$ with the projection $f: G \to \Ad(G)$.
   In fact, $\Ad(G)$ is also a split semi-simple $F$-group.
   Choose some maximal $F$-split torus $T \subseteq \Ad(G)$.
   By assumption $F$ contains the $\ell$-th roots of unity and since $T \cong \Gm^r$ over $F$,
   the elements of order $\ell$ in $T(F)$ form a finite group $\Theta$ isomorphic
   to $(\bbZ/\ell\bbZ)^r$. This group acts by $F$-rational inner automorphisms on $G$.
   Since $\Ad(G)$ is adjoint, the roots of $T$ acting on the Lie algebra $\Lie(G)$ span the group of characters
   $X(T)$ of $T$ (cf.~\cite[(2.23)]{BorelTits1972}).
   Any root can be considered a simple root for some system of simple roots (cf.~\cite[10.3]{Humphreys1972}) and
   thus cannot be in $\ell X(T)$. We see that no root vanishes entirely on $\Theta$ and
   we deduce that $\Lie(G)^\Theta = \Lie(T)$. This implies that the identity component $(G^\Theta)^0$ is a maximal $F$-split torus $T'$
   isogenous to $T$ via the projection $f$. Hence $G^\Theta$ is an extension of a finite group by~$T'$.
   
   Take $\Gamma_0$ to be a finite index $\Theta$-stable subgroup of $\Gamma$ such that $\Gamma_0 \cap G^\Theta \subset T'$.
   This is possible since $T'$ is of finite index in $G^\Theta$ (embed $G\to \GL_N$ and apply a suitable congruence condition away from $S$). 
   The closure $U_0$ of $\Gamma_0$ in $G(F \otimes_\bbQ \bbQ_p)$ is open and compact, since $\widetilde{G}$ has strong approximation 
   (cf.~Lemma \ref{lem:openclosure}).
   By Lemma \ref{lem:exU} there is an open, compact, normal, $\Theta$-stable, uniform pro-$p$ subgroup $U \subseteq U_0$.
   We define $\Gamma_n := P_n(U) \cap \Gamma_0$.
   The fixed point groups $\Gamma_n^\Theta$ are torsion-free arithmetic subgroups of the split torus $T'$ and therefore are free abelian
   of rank $r(|S|-1)$ by Dirichlet's $S$-unit theorem.
   Note that the free abelian group $\Gamma_n^\Theta$ of rank $r(|S|-1)$ has non-vanishing cohomology in degree $r(|S|-1)$.
   Moreover, since $\Gamma_n/\Gamma_{n+1}$ is an elementary abelian $p$-group for every $n$, we conclude that
   \begin{equation*}
       [\Gamma_1^\Theta:\Gamma_n^\Theta] \leq p^{(n-1)r(|S|-1)}. 
   \end{equation*}
   We define $\lambda = r(|S|-1)$ and $\alpha = 1$. Theorem \ref{thm:modGrowth} implies the claim  
    since the dimension of $G(F\otimes_\bbQ \bbQ_p)$ as $p$-adic analytic group
   is $[F:\bbQ]\dim(G)$ and the dimension of the fixed point group is $r[F:\bbQ] = r(s+2t)$.  
\end{proof}

\begin{remark}
   If $F$ is totally real, then the assumption $|S| \leq [F:\bbQ]$ restricts the theorem to arithmetic groups. Moreover, a totally real
   field only contains the roots of unity $\pm1$, so we can only speak about $\bbF_2$ cohomology.

   Note further that in the situation of Theorem \ref{thm:splitGroups} the real rank of the Lie group
   $G_\infty$ is $r(s+t)$. 
   If $S$ is the set of archimeden places and $\Gamma$ is an irreducible arithmetic group, it is known that the cohomology with $\bbC$-coefficients of $\Gamma$ below the degree $r(s+t)$ comes 
   only from the cohomology of the compact dual symmetric space (cf.~\cite[XIV, 2.2]{BorelWallach2000}).
   However, it seems that a similar result is not known for cohomology with mod-$\ell$ coefficients (and probably does not hold).
\end{remark}

\section{Uniform pro-$p$ groups and finite order automorphisms}\label{sec:prop}
In this section we prove some auxiliary results concerning pro-$p$ groups, finite order automorphisms and non-abelian cohomology.
Throughout $G$ denotes a pro-$p$ group.
The lower $p$-series of $G$ will be denoted by $P_i(G)$ (a general definition can be found in 1.15 in \cite{DdSMS1991}). 
We shall be mainly concerned with the case where $G$ is a uniform pro-$p$ group (see \cite{DdSMS1991}).
In this case the lower $p$-series is defined as
\begin{equation*}
  P_1(G) := G \quad\text{ and }\quad P_{i+1}(G) := \Phi(P_i(G)).
\end{equation*}
Here we use $\Phi$ to denote the \emph{Frattini subgroup}.

Let $\Theta$ denote a finite group which acts on $G$ by continuous automorphisms. We denote the action by upper left exponents.
Consider the first non-abelian cohomology set $H^1(\Theta, G)$. Recall that a $1$-cocycle for $\Theta$ with values in $G$ is
a function $a: \Theta \to G$ (with notation $a(s) = a_s$) such that $a_{st} = a_s \up{s}a_t$ for all $s,t \in \Theta$.
Two cocycles $a$ and $b$ are \emph{cohomologous},
 if there is some element $g \in G$ such that $a_s = g^{-1}b_s \up{s}g$ for all $s \in \Theta$.

\begin{lemma}\label{lem:H1vanishes}
  If the order of $\Theta$ is prime to $p$, then
  $H^1(\Theta,G) = \{1\}$.
\end{lemma}
\begin{proof}
   The claim is obvious if $G$ is a finite elementary abelian $p$-group. 
   If $G$ is a finite $p$-group, the claim can be proven by induction on the order of $G$.
   Suppose now that $G$ is a general pro-$p$ group and let $a$ be some $1$-cocycle with values in~$G$.
   For every open normal $\Theta$-stable subgroup $N \normal G$ we have
   $H^1(\Theta, G/N) = \{1\}$. Therefore the set
   \begin{equation*}
       K(N) = \{\:g \in G \:|\: \forall s\in \Theta \quad g a_s \up{s}g^{-1} \in N\:\}
   \end{equation*}
   is a \emph{non-empty} open compact subset of $G$.
   Since $G$ is compact the intersection of all the sets $K(N)$ is non-empty,
   where $N$ runs through the collection of open normal $\Theta$-stable subgroups. 
   Hence there is some element $g \in G$ such that $g a_s \up{s}g^{-1} \in N$ for all such $N$, and we deduce 
   $g a_s \up{s}g^{-1} = 1$.
\end{proof}

We will frequently need this lemma, in particular it is the main tool in the proof of the following fundamental proposition.

\begin{proposition}\label{prop:uniform}
 Let $G$ be a uniform pro-$p$ group.
 If $p$ does not divide the order of $\Theta$, then the fixed point group $G^\Theta$ is a uniform pro-$p$ group.
 Moreover we have $P_i(G^\Theta) = P_i(G)^\Theta$ for all $i \geq 1$.
 In particular,  $\dim_{\bbF_p} \bigl(G/\Phi(G)\bigr)^\Theta = \dim( G^\Theta )$.
\end{proposition}
\begin{proof}
   Consider the short exact sequence of groups 
   \begin{equation*}
      1 \longrightarrow \Phi(G) \longrightarrow G \longrightarrow G/\Phi(G) \longrightarrow 1.
   \end{equation*}
   Since $p \nmid |\Theta|$, we have $H^1(\Theta, \Phi(G)) = \{1\}$ and thus the sequence
   \begin{equation}\label{eq:seqXi}
      1 \longrightarrow \Phi(G)^\Theta \longrightarrow G^\Theta \longrightarrow \bigl(G/\Phi(G)\bigr)^\Theta \longrightarrow 1
   \end{equation}
   is exact. Note that $G/\Phi(G)$ is a finite dimensional $\bbF_p[\Theta]$-module which decomposes into isotypic components.
   By \eqref{eq:seqXi} we can find elements $g_1, \dots, g_r \in G^\Theta$ which form an $\bbF_p$-basis of the $1$-isotypic component in 
   $G/\Phi(G)$.
   We choose elements $u_1, \dots, u_n$ in $G$ which form a basis of the other isotypic components, in particular
   \begin{equation*}
       g_1, \dots, g_r, u_1, \dots, u_n
   \end{equation*}
    is an $\bbF_p$-basis of $G/\Phi(G)$.

    \textbf{Claim:} $G^\Theta  = \overline{\langle g_1\rangle}\cdots\overline{\langle g_r \rangle}$. \\
    In particular, this means that $G^\Theta$ is topologically generated by $g_1,\dots, g_r$.
    To prove the claim, we take some $g \in G^\Theta$. By 3.7 in \cite{DdSMS1991} we can write
    \begin{equation*}
        g = g_1^{\lambda_1} \cdots g_r^{\lambda_r} u_1^{\mu_1} \cdots u_n^{\mu_n}
    \end{equation*}
    for certain $p$-adic integers $\lambda_i, \mu_j \in \bbZ_p$.
    Since $g\Phi(G)$ is a fixed point in $G/\Phi(G)$, we see that the $\mu_j$ are all divisible by $p$.
    Further, all $g_i$ are fixed under $\Theta$, hence the element 
     $g' := u_1^{\mu_1} \cdots u_n^{\mu_n}$ is also in $G^\Theta$. Suppose there is some $j$ with $\mu_j \neq 0$.
    Then we can find an integer $k$ such that $g' \in P_{k}(G) \setminus P_{k+1}(G)$.
    However, the map $x \mapsto x^{p^{k-1}}$ defines an isomorphism of $\bbF_p[\Theta]$-modules 
   \begin{equation*}
       G/\Phi(G) \isomorph P_k(G)/P_{k+1}(G)
   \end{equation*}
    (here we use that $G$ is a uniform pro-$p$ group).
    So, by construction, $g'$ does not lie in the $1$-isotypic component since it is an element of the span of
     $u_1^{p^{k-1}}, \dots, u_n^{p^{k-1}}$. We deduce that $g' = 0$ and thus
    $G^\Theta = \overline{\langle g_1, \dots, g_r \rangle} = \overline{\langle g_1\rangle}\cdots\overline{\langle g_r \rangle}$.

      Given a pro-$p$ group $H$ we denote the subgroup generated by the $p^k$-th powers in $H$ by $H^{p^k}$.
     As a next step we show that $(G^\Theta)^{p^k} = (G^{p^k})^\Theta$ for all $k \geq 1$. This implies, in particular, that $(G^\Theta)^{p^k}$ is closed.
     The inclusion $(G^\Theta)^{p^k} \subseteq (G^{p^k})^\Theta$ is obvious. Conversely, let
     $g \in (G^{p^k})^\Theta$. By the Claim, applied to the group $P_{k+1}(G) = G^{p^k}$, we have
     $(G^{p^k})^\Theta = \overline{\langle g^{p^k}_1\rangle}\cdots\overline{\langle g^{p^k}_r \rangle} \subseteq (G^\Theta)^{p^k}$. 

     Finally we can show that $G^\Theta$ is a uniform pro-$p$ group. By the Claim the group $G^\Theta$ is topologically finitely generated.
     We show that $G^\Theta$ is powerful. To show this we assume $p\neq 2$, since the argument for $p=2$ is analogous.
     Since $\Phi(G) = G^p$, the exact sequence \eqref{eq:seqXi}, implies that $G^\Theta/(G^\Theta)^p$ is abelian.
     Now we know that $G^\Theta$ is powerful, thus we have $P_k(G)^\Theta = P_k(G^\Theta)$.
     Moreover, the short exact sequence
    \begin{equation*}
          1 \longrightarrow P_{k+1}(G) \longrightarrow P_k(G) \longrightarrow G/\Phi(G) \longrightarrow 1
    \end{equation*}
    yields the short exact sequence
     \begin{equation*}
          1 \longrightarrow P_{k+1}(G^\Theta) \longrightarrow P_k(G^\Theta) \longrightarrow \bigl(G/\Phi(G)\bigr)^\Theta \longrightarrow 1
    \end{equation*}
    by Lemma \ref{lem:H1vanishes}. Consequently, $G^\Theta$ is uniform and $ \dim_{\bbF_p} \bigl(G/\Phi(G)\bigr)^\Theta$ 
    is the dimension of~$G^\Theta$.
\end{proof}

\section{Proof of the main result}\label{sec:proof}
In this section we prove the main result. Let $p$ and $\ell$ be two distinct prime numbers and
let $\Gamma$ be a group of type \typ{VFP}.
Recall that $\Gamma$ is assumed to be a dense subgroup of a compact $p$-adic analytic group $G$.
Let $\Theta$ be a finite $\ell$-group of continuous automorphisms of $G$ and assume that $\Gamma$ is $\Theta$-stable.
For reasons of completeness we prove the following basic lemma.
\begin{lemma}\label{lem:exU}
 The $p$-adic analytic group $G$ has an open, normal, $\Theta$-stable, uniform pro-$p$ subgroup $U$ such that $U \cap \Gamma$
 is torsion-free. 
\end{lemma}
\begin{proof}
The compact p-adic analytic group $G \semidirect \Theta$ contains an open, normal, uniform pro-$p$ subgroup $U_0$
(see 9.36 in \cite{DdSMS1991}).
Define $U = U_0 \cap G$. The quotient $U_0/U$ is a finite $p$-group which maps injectively into the finite $\ell$-group $\Theta$. We conclude that 
$U/U_0 = \{1\}$. Thus $U = U_0$ is an open, normal, $\Theta$-stable, uniform pro-$p$ subgroup of~$G$.
Finally, $U$ is torsion-free (by 4.8 in \cite{DdSMS1991}), in particular $\Gamma \cap U$ is torsion-free. 
\end{proof}

We fix a subgroup $U \normal G$ as in Lemma \ref{lem:exU} and we consider the 
associated $p$-adic analytic tower
\begin{equation*}
    \Gamma_i := \Gamma \cap P_i(U).
\end{equation*}
We shall now prove Theorem \ref{thm:modGrowth} stated in the introduction.
\begin{proof}[Proof of Theorem \ref{thm:modGrowth}]
   The groups $\Gamma_n$ are torsion-free and so, since $\Gamma$ was assumed to be of type \typ{VFP},
   we find that $\Gamma_n$ is of type \typ{FP}. In particular the sums in the statement of the theorem are finite and well-defined.
   Given a $1$-cocycle $c$ of $\Theta$ with values in $\Gamma_n$, we can twist the action of $\Theta$ on $\Gamma_n$ by $c$.
   More precisely, $s \in \Theta$ acts by $\up{s|c}\gamma := c_s \up{s}\gamma c_s^{-1}$.
   The fixed points of the $c$-twisted action will be denoted by $\Gamma_n^{\Theta|c}$.  
   Note that for cohomologous cocycles $a$ and $b$ the fixed point groups of the corresponding twisted actions are isomorphic.
   The Smith theory argument in the introduction yields a
    slight variation of a theorem of Adem:
   \begin{equation}\label{eq:Adem}
      \sum_{i=r}^\infty \dim_{\bbF_\ell} H^i(\Gamma_n, \bbF_\ell) \geq \sum_{[c] \in H^1(\Theta, \Gamma_n)} \sum_{i=r}^\infty \dim_{\bbF_\ell} H^i(\Gamma^{\Theta|c}_n, \bbF_\ell).
   \end{equation}
   This is essentially Adem's Thm.~3.3 in \cite{Adem1996}. However, Adem only proves the case $r=0$, the 
   statement for $r > 0$ is obtained via a modification in his Thm.~1.5.
   One simply has to replace the Smith theory applied by Adem with the finer Theorem~4.4 of Floyd \cite{Floyd1952}
   (or Thm.~4.1 in \cite[Ch.III]{Bredon1972}).

   Let $c \in Z^1(\Theta, \Gamma_n)$ be a cocycle whose cohomology class $[c]$ lies in the kernel of the natural map
   $h_n: H^1(\Theta, \Gamma_n) \to H^1(\Theta, \Gamma_1)$. This means there is an element $a \in \Gamma_1$ such that
   \begin{equation*}
        c_s = a^{-1} \up{s}a \quad\text{ for all } s \in \Theta.
   \end{equation*}
    The inner automorphism $\inn(a): \Gamma \to \Gamma$ stabilizes the normal subgroups $\Gamma_j$ and satisfies
    \begin{equation*}
       \inn(a)(\up{s|c}\gamma) = ac_s\up{s}\gamma c_s^{-1}a^{-1} = \up{s}\inn(a)(\gamma)
    \end{equation*}
    for all $s \in \Theta$ and $\gamma \in \Gamma$.
    Consequently, $\Gamma_n^{\Theta|c} \cong \Gamma_n^\Theta$.
    From Adem's inequality \eqref{eq:Adem} we obtain
    \begin{align*}
        \sum_{i=r}^\infty \dim_{\bbF_\ell} H^i(\Gamma_n, \bbF_\ell) &\geq
      \sum_{[c] \in H^1(\Theta, \Gamma_n)} \sum_{i=r}^\infty \dim_{\bbF_\ell} H^i(\Gamma^{\Theta|c}_n, \bbF_\ell)\\
      &\geq \left|\ker(h_n)\right|\:\sum_{i=r}^\infty \dim_{\bbF_\ell} H^i(\Gamma^{\Theta}_n, \bbF_\ell) . 
    \end{align*}
    As a next step we determine the cardinality of $\left|\ker(h_n)\right|$.
    Since $\Gamma_1/\Gamma_n$ is isomorphic to $U/P_n(U)$, we obtain an exact sequence of pointed sets
    \begin{equation*}
        0 \longrightarrow \Gamma_n^\Theta \longrightarrow \Gamma_1^\Theta \longrightarrow \bigl(U/P_n(U)\bigr)^\Theta \longrightarrow \ker(h_n) \longrightarrow 0.
    \end{equation*}
    By Proposition \ref{prop:uniform} and Lemma \ref{lem:H1vanishes} we obtain $\bigl(U/P_n(U)\bigr)^\Theta \cong U^\Theta / P_n(U^\Theta)$ and we deduce 
    \begin{equation*}
     \left|\ker(h_n)\right| = [\Gamma_1^\Theta:\Gamma^\Theta_n]^{-1} \bigl|\bigl(U/P_n(U)\bigr)^\Theta\bigr| = [\Gamma_1^\Theta:\Gamma^\Theta_n]^{-1} p^{(n-1)\dim(G^\Theta)}.
    \end{equation*}
    Finally we conclude
    \begin{align*}
        \sum_{i=r}^\infty \dim_{\bbF_\ell} H^i(\Gamma_n, \bbF_\ell) &\gg
         [\Gamma_1^\Theta:\Gamma^\Theta_n]^{\alpha-1} p^{(n-1)(\dim(G^\Theta)-\alpha\lambda)}\\ &\gg p^{(n-1)\alpha(\dim(G^\Theta)-\lambda)} 
       \: = \: [\Gamma_1:\Gamma_n]^{d\alpha}.
    \end{align*}
    Here we use $\alpha \leq 1$ and the inequality $[\Gamma_1^\Theta:\Gamma^\Theta_n] \leq |\bigl(U/P_n(U)\bigr)^\Theta\bigr|$.
\end{proof}

\section{Cohomology with rational coefficients} \label{sec:rational}

Let $\Gamma$ be a torsion-free arithmetic group and consider the rational cohomology groups $H^*(\Gamma,\bbQ)$.
In fact,  $\dim_\bbQ H^i(\Gamma,\bbQ) \leq \dim_{\bbF_\ell} H^i(\Gamma,\bbF_\ell)$ for every prime number~$\ell$. 
It is therefore a natural question to ask whether
the cohomology growth predicted by Theorem \ref{thm:modGrowth} (and the given applications) is already valid for cohomology with rational coefficients.
If this does not hold, then one could deduce that the cohomology groups $H^*(\Gamma,\bbZ)$ contain many torsion elements of order $\ell$.
However, we expect that for arithmetic groups the growth rates should already be valid for the cohomology with rational coefficients.
In this section we
prove a result in this direction under certain conditions. In particular, $\Theta$ needs to be a finite cyclic group.
The underlying powerful tool is Rohlfs' method for the computation of Lefschetz numbers (cf.~\cite{Rohlfs1990}).

It is convenient to make the following definition.
\begin{definition}
  Let $H$ be a reductive linear algebraic $\bbQ$-group.
  We say that $H$ has the \emph{non-vanishing property} if 
  for some (and hence for all) torsion-free arithmetic subgroup $\Gamma \subseteq H(\bbQ)$ the
  Euler characteristic $\chi(\Gamma)$ is non-zero.
\end{definition}
For instance $H$ has the non-vanishing property if one of the following holds
\begin{enumerate}
 \item $H(\bbR)$ is semi-simple and has DS.
 \item $H(\bbR)$ is compact.
 \item $H$ is a $\bbQ$-split torus.
\end{enumerate}
The non-vanishing property is stable under isogenies, 
since isogenies map torsion-free arithmetic groups isomorphically to arithmetic groups (see Thm.~8.9 in \cite{Borel1969}).
 Moreover, it is preserved under direct products.
For example this shows that $\GL_2 / \bbQ$ has the non-vanishing property.

Let $G$ be a semi-simple connected linear algebraic
$\bbQ$-group and let $\tau$ be a $\bbQ$-rational automorphism of $G$ of finite order.

\begin{theorem}\label{thm:Qgrowth}
  Let $p$ be a prime number which does not divide the order of $\tau$.
  Suppose that the simply connected covering of $G$ has strong approximation and assume
  that the fixed point group $G^\tau$ has the non-vanishing property. 
  Then every arithmetic subgroup $\Gamma \subset G(\bbQ)$ contains a cofinal $p$-tower 
  $(\Gamma_n)_{n=1}^\infty$ such that
   \begin{equation*}
       \sum_{i=0}^\infty\dim H^i(\Gamma_n, \bbQ) \gg [\Gamma:\Gamma_n]^{d}
   \end{equation*}
   with $d = \frac{\dim G^\tau}{\dim G}$ as $n$ tends to infinity.
\end{theorem}
\begin{proof}
    Recall that a $\tau$-stable arithmetic subgroup $\Gamma' \subseteq G(\bbQ)$ is
   \emph{small enough} in the sense of Rohlfs \cite[Def.~2.7]{Rohlfs1990} if it is torsion-free, contained in $G(\bbR)^0$, and the image of $H^1(\tau,\Gamma')$
   is trivial in $H^1(\tau,G(\bbQ_q))$ for every prime number $q$.
   
   Choose a subgroup $\Gamma_1 \subseteq \Gamma$ such that the following conditions hold
   \begin{enumerate}
    \item $\Gamma_1$ is $\tau$-stable,
    \item  $\Gamma_1$ is \emph{small enough}, and
    \item the closure of $\Gamma_1$ in $G(\bbQ_p)$ is an open compact uniform pro-$p$ subgroup.
   \end{enumerate}
   It is easy to see that such a group $\Gamma_1$ exists: Take $\Gamma_0$ to be small enough and $\tau$-stable, which is possible due to 
   Prop.~2.8 in \cite{Rohlfs1990}, then consider the closure of $\Gamma_0$ in $G(\bbQ_p)$ and use Lemma \ref{lem:exU}.
   We denote the closure of $\Gamma_1$ in $G(\bbQ_p)$ by $U$. Further we set $\Gamma_i = P_i(U)\cap\Gamma_1$.
   Consider the Lefschetz number of $\tau$ on $\Gamma_n$ defined by
   \begin{equation*}
       \calL(\tau, \Gamma_n) = \sum_{i=0}^\infty (-1)^i \Tr(\tau | H^i(\Gamma_n,\bbQ)).
   \end{equation*}
    For every cocycle $c \in Z^1(\tau, \Gamma_n)$ we define the $c$-twisted $\tau$-action
    on $\Gamma_n$ by $\up{\tau|c}\gamma := c_\tau \up{\tau}\gamma c_\tau^{-1}$ (just as in the proof of Theorem \ref{thm:modGrowth}).
    The fixed points under the $c$-twisted action will be denoted by $\Gamma_n^{\tau|c}$.
     By Rohlfs' method we know that the Lefschetz number is given by a sum of Euler characteristics, namely
    \begin{equation*}
      \calL(\tau, \Gamma_n) = \sum_{[c] \in H^1(\tau,\Gamma_n)} \chi(\Gamma_n^{\tau|c})
    \end{equation*}
    (see \cite{Rohlfs1990} or \cite{RohlfsSpeh1989}).
    In general the occuring Euler characteristics can have different signs, however, it is a result due to Rohlfs-Speh
    that this does not happen when $\Gamma_n$ is \emph{small enough} in their sense (cf.~Lem.~2.5 in \cite{RohlfsSpeh1989}).
    Since we need a slight generalization of their result, we include it as Lemma \ref{lem:RohlfsSpehMod4} below.

    Let $h_n: H^1(\tau, \Gamma_n) \to H^1(\tau,\Gamma_1)$ be the canonical map.
    We can proceed just like in the proof of Theorem~\ref{thm:modGrowth} to obtain:
   \begin{equation*}
       \left|\calL(\tau,\Gamma_n)\right| \geq \sum_{[c] \in \ker(h_n)}  \left|\chi(\Gamma_n^{\tau|c})\right|.
   \end{equation*}
   Note that $G^\tau \cong G^{\tau|c}$ and $\Gamma_n^\tau \cong \Gamma_n^{\tau|c}$ for every class $c \in \ker(h_n)$.
   Since $G^\tau$ has the non-vanishing property, we know that the Euler characteristics are non-zero 
   and they grow linearly with the index $[\Gamma^\tau_1:\Gamma^\tau_n]$.
   The claim follows as in Theorem~\ref{thm:modGrowth}.
\end{proof}

  \begin{remark}
      Using Theorem \ref{thm:Qgrowth} one can deduce an analogue of Theorem~\ref{thm:basechange} 
      for rational cohomology and cyclic base change.
     One can use this to give a short proof of the main theorem of \cite{KSchwermer2012}.
     Note however, that the precise Lefschetz number given in \cite{KSchwermer2012} contains much more information.

     Note that in Theorem~\ref{thm:Qgrowth} there is no parameter $r$ as in Theorem~\ref{thm:modGrowth}.
     This makes it difficult to locate the cohomology growth in a specific degree.
     Moreover, Theorem~\ref{thm:Qgrowth} only works for actions of finite cyclic groups and it involves stronger restrictions
     on the fixed point groups. This makes it impossible to use Theorem~\ref{thm:Qgrowth} recursively, as we did in Example \ref{ex:SL}.
  \end{remark}
 
\begin{example}
   Let $G = \SL_3 / \bbQ$ and let $\tau$ denote the inner automorphism defined by conjugation with
   \begin{equation*}
      I_{2,1}  = \diag(-1,-1,1).
   \end{equation*}
  The fixed point group $G^\tau$ is isomorphic to $\GL_2/\bbQ$. Hence it has the non-vanishing property.
  We find $d = \frac{1}{2}$ since $\dim G = 8$ and $\dim G^\tau = 4$.
  So every arithmetic subgroup of $\SL_3$ contains a cofinal $p$-tower (for every odd prime $p$)
  with some Betti number growing faster than the square root of the index. 
\end{example}

In the remaining part we state and prove the main lemma needed in the proof of Theorem \ref{thm:Qgrowth}. The result is essentially due to Rohlfs-Speh (cf.~Lem.~2.5 in \cite{RohlfsSpeh1989}).
We include a proof since it was only obtained for Cartan-like involutions in \cite{RohlfsSpeh1989}. 
Rohlfs indicated  \cite[(2.9)]{Rohlfs1990} that the 
result holds in general.

The sign of the Euler characteristic of a torsion-free arithmetic subgroup of a reductive algebraic $\bbQ$-group $H$  depends only on $H$.
Recall how to determine the sign.
 We can ignore the centre of $H$ since a finite rank free abelian group has Euler characteristic $0$ or $1$.
 Let $H_{s}$ denote the derived group of $H^0$ and take any maximal compact subgroup $K^s_H$ of $H_{s}(\bbR)$. We define $q_H := \dim H_{s}(\bbR)/K^s_H$.
 The following is an easy consequence of 
 Harder's Gau\ss-Bonnet Theorem: \emph{If 
 $\chi(\Gamma_H) \neq 0$ for some 
 torsion-free arithmetic subgroup $\Gamma_H \subset H(\bbQ)$, then $q_H$ is even and $\chi(\Gamma_H)$ has sign $(-1)^{q_H/2}$ for all $\Gamma_H$} (see \cite{Harder1971}).

As before $G$ denotes a semi-simple connected linear algebraic
$\bbQ$-group and $\tau$ is a $\bbQ$-rational finite order automorphism of $G$.

\begin{lemma}[Rohlfs-Speh]\label{lem:RohlfsSpehMod4}
   Let $b \in Z^1(\tau,G(\bbQ))$ be a cocycle such that $b$ defines the trivial class in $H^1(\tau,G(\bbQ_p))$ for every prime number $p$.
   Assume that $G^\tau$ has the non-vanishing property and let $\delta$ denote the sign of the Euler characteristic of arithmetic subgroups of $G^\tau$.
   Every torsion-free arithmetic subgroup $\Gamma$ of the twisted fixed point group $G^{\tau|b}$ satisfies
   \begin{equation*}
          \delta\cdot\chi(\Gamma) \geq 0.
   \end{equation*}
\end{lemma}
\begin{proof}
      Note that the fixed point groups $G^\tau$ and $G^{\tau|b}$ are reductive (cf.~\cite{Jacobson1962}).
      For any cocycle $c$ in $Z^1(\tau,G(\bbR))$ let 
      $q_{\tau|c}$ denote the number $q_{G^{\tau|c}}$ for the reductive group $G^{\tau|c}$ as defined above.
      We show that $q_\tau \equiv q_{\tau|b} \bmod 4$. (We do not exclude the possibility that arithmetic subgroups of $G^{\tau|b}$ have Euler characteristic zero.)

     We start with a useful observation.
     Let $\LG = \LG_\bbQ$ be the Lie algebra of $G$ over $\bbQ$ and let $B$ denote the Killing form on~$\LG$.
     We shall always write $B$ for the restriction of $B$ to some subspace.
     Let $F/\bbQ$ be any field extension and let $c \in Z^1(\tau,G(F))$ be a cocycle. We write $\LG_F := \LG \otimes_\bbQ F$ and we put
     $\sigma = \tau|c$. The fixed point space $\LG_F^\sigma$ is a reductive subalgebra and the restriction of $B$ to this space is non-degenerate.
     This follows from the fact that the action of $\sigma$ preserves the Killing form. 
     Moreover, if we decompose $\LG_F^\sigma = \LZ_{F,\sigma} \oplus \LG'_{F,\sigma}$ into the centre $\LZ_{F,\sigma}$ and the semi-simple part $\LG'_{F,\sigma} = [\LG_F^\sigma,\LG_F^\sigma]$,
     then these spaces are orthogonal with respect to $B$. 
     Finally, suppose $d \in Z^1(\tau, G(F))$ is a cocycle in the class of $c$, then there is $g \in G(F)$ such that $d = g^{-1}c\up{\tau}g$ and $\Ad(g)$ defines an isometry
     between the quadratic spaces $(\LG'_{F,\tau|d},B)$ and $(\LG'_{F,\tau|c},B)$.
     We conclude that there is a well-defined map 
     \begin{equation*}
          \rho_F: H^1(\tau,G(F)) \to W(F)
     \end{equation*}
     defined by $[c] \mapsto (\LG'_{F,\tau|c},B)$ into the Witt ring $W(F)$ of $F$.

     The cocycle $b$ is by assumption in the trivial class in $H^1(\tau,G(\bbQ_p))$ for every prime number $p$. Hence the quadratic $\bbQ$-spaces $(\LG'_\tau,B)$ and $(\LG'_{\tau|b},B)$
     are isometric over all $p$-adic fields. It follows from Weil's product formula (see Ch.~5.~Cor.8.2 in \cite{Scharlau1985})
     that the signatures of the quadratic $\bbR$-spaces $(\LG'_{\bbR,\tau},B)$ and $(\LG'_{\bbR,\tau|b},B)$
     are congruent modulo $8$. Here the signature of a quadratic $\bbR$-space is the difference of
     the dimensions of a maximal positive definite and a maximal negative definite subspace. 
     For $c \in Z^1(\tau,G(\bbR))$ we denote the signature of $(\LG'_{\bbR,\tau|c},B)$ by $\sign(\tau|c)$.
    
     Let $K \subset G(\bbR)$ be a $\tau$-stable maximal compact subgroup and let $c$ be a cocycle in $Z^1(\tau,G(\bbR))$. We study the number $q_{\tau|c}$ associated with $G^{\tau|c}$ or, more precisely, with
     its derived group $G_s^{\tau|c}$. Since the map 
     $H^1(\tau,K) \to H^1(\tau,G(\bbR))$ is bijective (cf.~\cite{AnWang2008}), we assume $c \in Z^1(\tau, K)$.
     In particular, $K$ is stable under $\tau|c$. Consider the Cartan decomposition
     \begin{equation*}
           \LG_\bbR = \LK \oplus \LP_0
     \end{equation*}
     where $\LK$ denotes the Lie algebra of $K$.
     The semi-simple part $\LG'_{\bbR,\tau|c}$ of the fixed point algebra inherits this decomposition
      \begin{equation*}
          \LG'_{\bbR,\tau|c} = (\LG'_{\bbR,\tau|c}\cap\LK)\: \oplus \: (\LG'_{\bbR,\tau|c}\cap\LP_0).
     \end{equation*}
     The Killing form $B$ is negative definite on $\LK$ and positive definite on $\LP_0$.
     Moreover, $\LG'_{\bbR,\tau|c}\cap\LK$ is the Lie algebra of the maximal compact subgroup $G^{\tau|c}_s(\bbR)\cap K$ of $G^{\tau|c}_s(\bbR)$.
     We deduce that
     \begin{equation*}
        2q_{\tau|c} = 2 \dim_\bbR(\LG'_{\bbR,\tau|c}\cap\LP_0) = \dim_\bbR(\LG'_{\bbR,\tau|c}) + \sign(\tau|c).
     \end{equation*}
     The terms on the right hand side only depend on the class of $c$. Further, the dimensions $\dim (\LG'_{\tau})$ and $\dim (\LG'_{\tau|b})$ agree, since $b$ is in the trivial class
    over all $p$-adic fields.
    Thus we have proven $2q_{\tau} \equiv 2q_{\tau|b} \bmod 8$.
\end{proof}

\bibliography{lblibrary}{}
\bibliographystyle{amsplain}
\end{document}